\documentclass{amsart}
\usepackage[T1]{fontenc}
\usepackage[utf8]{inputenc}
\usepackage{amsmath,amsthm,amsfonts,amscd,amssymb,eucal,latexsym,mathrsfs}
\usepackage{stmaryrd}
\usepackage{enumerate}
\usepackage{hyperref}
\usepackage[all]{xy}
\usepackage{etoolbox}

\usepackage{tikz,tikz-cd}
\usetikzlibrary{shapes.geometric}
\usetikzlibrary{arrows}

\usetikzlibrary{positioning}
\usepackage{float}
\usepackage{MnSymbol}
\usetikzlibrary{matrix}

\newcommand{\toc}{\tableofcontents}

\theoremstyle{plain}
\newtheorem{theorem}{Theorem}[section]
\newtheorem*{theorem*}{Theorem}

\newtheorem*{corollary*}{Corollary}

\newtheorem{lemma}[theorem]{Lemma}

\newtheorem{proposition}[theorem]{Proposition}

\newtheorem{claim}[theorem]{Claim}
\theoremstyle{definition}
\newtheorem{remark}[theorem]{Remark}

\newtheorem{question}[theorem]{Question}

\newtheorem{definition}[theorem]{Definition}
\newtheorem*{definition*}{Definition}

\usetikzlibrary{calc,graphs}
\usepackage{xcolor}
\usetikzlibrary{arrows,decorations.pathmorphing}

\newcommand{\p}{\varphi}

\newcommand{\e}{\varepsilon}

\newcommand{\IR}{\mathbb{R}}

\newcommand{\ZI}{\mathbb{Z}}
\newcommand{\IN}{\mathbb{N}}

\DeclareMathOperator{\del}{\partial}

\DeclareMathOperator{\impl}{\Rightarrow}

\newcommand{\ts}{\textsection}

\newcommand{\dirlim}{\varinjlim}

\newcommand{\Aut}{\mathrm{Aut}}
\newcommand{\Id}{\mathrm{Id}}

\title{Sidon sequences and nonpositive curvarture}
\author{Sylvain Barr\'e}
\author{Mika\"el Pichot}
\address{Sylvain Barr\'e, UMR 6205, LMBA, Université de Bretagne-Sud,BP 573, 56017, Vannes, France}\email{Sylvain.Barre@univ-ubs.fr}
\address{Mika\"el Pichot, McGill University, 805 Sherbrooke St W., Montr\'eal, QC H3A 0B9, Canada}\email{pichot@math.mcgill.ca}
\begin{document}

\setcounter{tocdepth}{1}

\begin{abstract} 

A sequence $a_0<a_1<\ldots<a_n$  of nonnegative integers  is called a Sidon sequence if the sums of pairs $a_i+a_j$ are all different. In this paper we construct CAT(0) groups and spaces from Sidon sequences. The arithmetic condition of Sidon is shown to be equivalent to  nonpositive curvature, and the number of ways to represent an integer as an alternating sum of triples $a_i-a_j+a_k$ of integers from the Sidon sequence, is shown to determine the structure of the space of embedded flat planes in the associated CAT(0) complex. 
\end{abstract}

\maketitle

\section{Introduction}

A sequence $a_0 < a_1<\ldots <a_n$ of nonnegative integers is called a Sidon sequence if the sums of pairs $a_i+a_j$ (for  $i\leq j$) are pairwise different. Here we  assume that $a_0=0$. An example is $0, 2, 7, 8, 11$. 

Every Sidon sequence $a_0=0 < a_1<\ldots <a_n$ can be  extended by letting $a_{n+1}$ be  the smallest integer such that $a_0<\ldots <a_{n+1}$ satisfies the condition of Sidon. Starting at $a_0=0$, this process defines an infinite Sidon sequence
\[
0, 1, 3, 7, 12, 20, 30, 44, 65, 80, 96,...
\]
 called the Mian--Chowla sequence. 

Sidon, in connection with his investigations on Fourier series,  was interested in estimates of the number of terms not exceeding $N$ a Sidon sequence can have. In general, there are at most $<<N^{1/2}$ terms  in the interval $[0,N]$, and the Mian--Chowla sequence, for example, which verifies $a_n\leq n^3$, contains $>>N^{1/3}$ terms in this interval (see \cite{S}). Erd\"os conjectured that for every $\e>0$, there should exist denser infinite Sidon sequences with $>>N^{1/2-\e}$ terms. This problem is well studied. We quote here the well--known theorem of Ruzsa \cite{Ruzsa} which constructs sequences with 
$>>N^{\gamma+o(1)}$ terms, where $\gamma=\sqrt 2-1$.  
For finite  sequences, a famous theorem of Singer \cite{ErdosSinger,Si} in projective geometry provides,  for any fixed $\e>0$,  Sidon sequences with $> N^{1/2}(1-\e)$ terms for any large $N$. This  uses the fact that the asymptotic ratio of consecutive primes is 1, can be improved by using a sharper estimate for $p-p'$ where $p$ and $p'$ are consecutive primes (see \cite[\ts 2.3]{HR}). 
For the purpose of the present paper, we shall require the stronger condition that the sums of pairs $a_i+a_j$ ($i\leq j$) are all  different  modulo some integer $N\geq 2$. Such a sequence is called a \emph{Sidon sequence modulo $N$}. Clearly, every finite Sidon sequence is a Sidon sequence modulo $N$ for every  large enough integer $N$.

The present  paper  describes a connection between Sidon sequences and nonpositive curvature.  
 The basic idea is to associate with a Sidon sequence $a_0 < a_1<\ldots <a_n$ modulo $N$ a group $G$ acting geometrically a 2-complex $X$, whose geometric properties depend on the arithmetic properties of the sequence $a_0 < a_1<\ldots <a_n$, as follows: 
\begin{enumerate}
\item the sum of pairs condition  of Sidon ensures that the ambient space $X$ is nonpositively curved;
\item the alternating sums of triples $a_i-a_j+a_k$ of elements of the Sidon sequence  regulates the structure of embedded flat planes in the space $X$.    
\end{enumerate} 

\noindent We shall refer to \cite{BH} for nonpositive curvature and CAT(0) spaces. 

Before we state our main theorem,  we  indicate how  the structure of  embedded flat planes can be related to Sidon sequences. 
In some CAT(0) 2-complexes, including the ones defined in the present paper, the space of embedded flat planes  can encoded by means of a ring puzzle problem  \cite{autf2puzzles}.  A ring puzzle is a tessellation of the Euclidean plane obtained by using a (finite) set of shapes which are constrained locally by a set of conditions around every vertex. In this context, these local conditions are appropriately described by a (finite) set of  length $2\pi$ circles called the rings. Every ring is labeled and specifies unambiguously  the shapes that can be used in the neighbourhood of the given vertex. Here
we require ring puzzles in which all the shapes are equilateral triangles. 
We are given a set of $n+1$ equilateral triangles, labeled by integers in $\{0,\ldots, n\}$, together with a set of rings of the form:
 
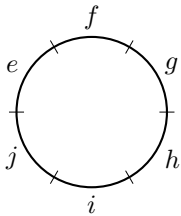
\begin{figure}[H]
 \begin{tikzpicture}
    \tikzset{narrow/.style={inner sep=3pt}}
    \draw [thick ] (0,0) circle[radius=1cm];
    \foreach \ang/\lab [evaluate=\ang as \anc using \ang+180]
        in {30/{$g$}, 90/{$f$}, 150/{$e$}, -150/{$j$}, -90/{$i$},-30/{$h$}}
        \draw [ultra thin] (\ang+30:.9)--(\ang+30:1.1) node{} (\ang:1) node[narrow, anchor=\anc]{\lab};
    \end{tikzpicture}
\end{figure}
\noindent where $e,f,g,h,i,j\in \{0,\ldots, n\}$ which specify which triangle label can be used around a vertex. The local condition ensures that the labels on the triangles and links are consistent.

Given these remarks, the following theorem  describes a connection between Sidon sequences, nonpositive curvature, and the structure of flats in the associated complexes.
The construction offers some freedom. Thus, one may use more than one Sidon sequence (here we shall have three possibly different sequences, all of the same length $n\geq 2$), and one may twist the construction by using permutations in the symmetric group $S_n$.

\begin{theorem}\label{T - main}
To every  integer $n\geq 2$, every triple of increasing sequences
\[
0\leq a_0^j < a_1^j<\ldots <a_n^j, \ \ j=1,2,3,
\]
every triple of integers $r_j$, $j=1,2,3$,
such that for every $j\in \{1,2,3\}$, the sequence $a_0^j < a_1^j<\ldots <a_n^j$ is a Sidon sequence modulo $r_j$, and every triples of bijections
\[
\sigma_j\colon \{a_0^j,\ldots, a_n^j\} \to \{0,\ldots , n\}, \ \ j=1,2,3
\] 
is associated a countable group $G$ acting properly and totally discontinuously on a CAT(0) complex $X$ of dimension 2 with compact quotient, and a $G$-invariant labelling $X^2\to \{0,\ldots,n\}$ of the faces of $X$, such that the following are \textbf{equivalent} for a flat plane with face labelled by the elements of $\{0,\ldots,n\}$: 

\begin{enumerate}
\item[(a)] $\Pi$ embeds in $X$ in a label preserving way
\item[(b)] $\Pi$ is the solution, with respect to the induced labelling, of a triangle ring puzzle problem with  rings of the form: 
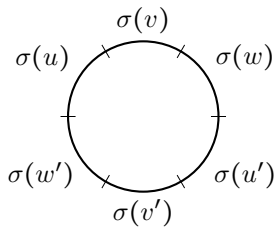
\begin{figure}[H]
 \begin{tikzpicture}
    \tikzset{narrow/.style={inner sep=3pt}}
    \draw [thick] (0,0) circle[radius=1cm];
    \foreach \ang/\lab [evaluate=\ang as \anc using \ang+180]
        in {30/$\sigma(w)$, 90/$\sigma(v)$, 150/$\sigma(u)$, -150/$\sigma(w')$, -90/$\sigma(v')$,-30/$\sigma(u')$}
        \draw [ultra thin] (\ang+30:.9)--(\ang+30:1.1) node{} (\ang:1) node[narrow, anchor=\anc]{\lab};
    \end{tikzpicture}
\end{figure}
\noindent such that $\sigma=\sigma_j$, $k=\tau_j(0)$, $l=\tau_j(1)$, $j\in \{1,2,3\}$,  for every triples $(a,b,c)$ and $(a',b',c')$ of elements of $\{a_0^j,\ldots, a_n^j\}$ with equal alternating sum
\[
a-b+c=a'-b'+c' \mod r_j
\]
such that $a\neq b\neq c$, $a'\neq b'\neq c'$ and $(a,b,c)\neq (a',b',c')$.
\end{enumerate}
 \end{theorem} 
 
 In the simplest situation, we shall introduce, for every Sidon sequence $a_0=0,\ldots, a_n$, a CAT(0) complex $X(a_0,\ldots, a_n)$,  called \emph{the modular complex of the Sidon sequence $a_0,\ldots, a_n$} (see Definition \ref{D- modular complex}).  It follows by Theorem \ref{T - main} that the automorphism group of $X(a_0,\ldots, a_n)$ is transitive on the vertex set.  These complexes appear to be new geometric objects for the most part, and we shall describe some examples  (old and new) in \ts\ref{S - examples} below.

Informally, Theorem \ref{T - main} \emph{reduces the description of embedded flat planes in the resulting complexes (including in modular complexes) to the resolution of a paving problem in $\IR^2$}. The latter problem is purely 2-dimensional in nature and has a number theoretic component, since the rings are obtained from the  representations of a modular integer as an alternating sum of triples from the given Sidon sequences.

Assuming that the Sidon sequences $(a_j^j)$ and the integers $r_j$ are fixed, we note that the  choice of the $\sigma_j$'s and $\tau_j$'s generates a priori $(n+1)!^3$ distinct groups and spaces associated with these data. The associated 2-dimensional paving problems depends on the choice of these elements a priori, and so does the space of embedded flats in $X$.

\bigskip

The paper is structured as follows.
The proof of Theorem \ref{T - main} has been divided into five steps,  occupying \ts \ref{S - graph S} to \ts\ref{S - puzzles embed}. It is shown in \ts\ref{S - graph S}  that the nonpositive curvature condition on  links (namely, having girth at least $2\pi$)  is equivalent to the arithmetic condition of Sidon. A new class of complexes, called completely homogeneous complexes, is introduced in \ts\ref{S - homogeneous}. The modular complexes are example of completely homogeneous complexes. In \ts\ref{S- signs} we prove a weaker version of Theorem \ref{T - main} in which has an additional ``sign redundancy''; the latter is removed in \ts \ref{S - Polarity}, by showing that the complexes in Theorem \ref{T - main} always admit sufficiently many ``polarities''; in particular, the various constructions in \ts\ref{S- signs} are in fact pairwise isomorphic. It is then possible to establish the equivalence stated in Theorem \ref{T - main} in \ts\ref{S - puzzles embed}.
The last section of the paper, \ts\ref{S - examples}, concludes with some examples and applications.

\bigskip
\bigskip

\noindent\textbf{Acknowledgment.}   The authors are partially funded by NSERC Discovery Fund 234313. 

\bigskip


\toc

\section{Sidon sequences}\label{S - graph S}

We refer to \cite{HR, O} for surveys on finite and infinite Sidon sequences; here we begin the proof of Theorem \ref{T - main}.

The first step in the proof is to relate the arithmetic condition of Sidon to a geometric condition on  spaces (i.e., nonpositive curvature via the link condition \cite{BH}).

Let $a_0=0<a_1<\ldots <a_n$ be a sequence of positive integers. We consider the graph $S$ on $\ZI$ in which $n+1$ edges are issued from every even vertex with an increment of $2a_r-1$, for every $r\in [0,\ldots, n]$. We write $S=S(a_1,\ldots ,a_n)$ to indicate the dependency in the sequence of integers.

\begin{theorem}
	Let $a_0=0<a_1<\ldots <a_n$ be an increasing sequence of positive integers and let $S=S(a_1,\ldots ,a_n)$. Let $N\geq 1$ be an integer. The following are equivalent:
	\begin{enumerate}
		\item the sequence $a_0,a_1,\ldots ,a_n$ is a Sidon sequence modulo $N$;
		\item the quotient of $S$ by the action of $2N\ZI$ by translations has girth at least 6.
	\end{enumerate}	
\end{theorem}

\begin{proof}
	Note that  the condition of Sidon, that the sums of pairs 
	\[
	a_i+a_j\  \ (i\leq j)
	\] 
	are all different modulo $N$, is verified if and only if the following alternating sums do not vanish
	\[
	a-b+c-d\neq 0\mod N
	\]
	for every quadruple $a\neq b\neq c\neq d\neq a$ chosen from the set $\{a_0,a_1,\ldots ,a_n\}$.
	
	We write $S_N$ for the quotient graph of $S$ by the translation action of $2N\ZI$.	
	By definition, $S_{N}$ is a bipartite graph.

	Suppose that
	\[
	a-b+c-d\neq 0\mod N
	\]
	for every $a\neq b\neq c\neq d\neq a$ as above. We claim that $S_{N}$ has no repeated edges. 
	Indeed, a repeated edge corresponds arithmetically to the the existence of distinct integers  $a\neq b$ such that $2a-1=2b-1\mod 2N$. Then 
	\[
	a-b+a-b=0\mod N.
	\] 
	This contradicts the fact that the condition $a\neq b\neq c\neq d\neq a$ is verified when $a\neq b$, $c=a$, $d=b$.
	
	Next we claim that $S_N$ contains no square. 
	
	Observe that a square in $S_N$ corresponds arithmetically to the choice of four elements   $a,b,c,d$ in the Sidon sequence, such that 
	\[
	a\neq b\neq c\neq d\neq a,
	\]
	 which indicates that consecutive edges in the square are distinct, and
	\[
	(2a-1)-(2b-1)+(2c-1)-(2d-1)=0\mod 2N,
	\]
	which indicates that the edges form a square. 
	However, the equality
	\[
	2(a-b+c-d)= 0\mod 2N
	\] 
again contradicts the Sidon condition modulo $N$. 
	
	This establishes that the condition of Sidon translates geometrically into the absence of (possibly degenerate) squares in the graph $S_N$.

	Since the graph $S_N$ is bipartite, it contains no cycle of odd length. This implies that the girth of $S_{N}$ is  $\geq 6$.

	Conversely, suppose that $S_{N}$ has girth at least 6. Since $S$ contains no square, it follows that
	\[
	(2a-1)-(2b-1)+(2c-1)-(2d-1)\neq 0\mod 2N
	\]
	where $a\neq b\neq c\neq d\neq a$. This implies that $a-b+c-d\neq 0\mod N$. Therefore, the Sidon condition holds modulo $N$.
\end{proof}

We refer to \cite{BP} for an earlier use of Sidon sequences in graph theory (in a different context).

\section{Completely homogeneous colourings of cell complexes}\label{S - homogeneous}

The second step constructs a complex of dimension 2, together with a homogeneous colouring of its cells in every dimension, assuming that  a similar homogeneous colouring exists in dimension 1. 

We shall first introduce some useful definitions.

\begin{definition}  A \emph{complete colouring} of an $n$-complex $X$ is an assignment of a colour to every cell of $X$ in such a way that any two incident cells of a given dimension have different colours. 
\end{definition}

More precisely, if $C$ is a set, a complete colouring of $X$ (with colours in $C$) is a map $\delta\colon X\to C$
such that 
\begin{enumerate}
\item[(a)] $\delta(e)\neq \delta(f)$ if  $e\neq f\in X^k$ and $e,f\subset g$ where $g\in X^{l}$ for some $l>k$; and,
\item[(b)] $\delta(e)\neq \delta(f)$ if  $e\neq f\in X^k$ and $g\subset e,f$ where $g\in X^{l}$ for some $l<k$.
\end{enumerate}
We call $(c_0,\ldots, c_n)$-colouring of $X$  a complete colouring $\delta$ such that   $c_k=|\delta(X_k)|$.

\begin{definition} If $(X,\delta)$ and $(X',\delta')$ are completely coloured complexes, a coloured homomorphism $f\colon X\to X'$ of $n$-complexes is a coloured homomorphism if $\delta=\delta'\circ f$. 
\end{definition}

We let $\Aut(X,\delta)$ denote the group of coloured automorphisms of $X$. We are interested in the following.

\begin{definition}   An $n$-complex $X$ is \emph{completely homogeneous} if it admits a complete colouring $\delta$ such that $\Aut(X,\delta)$ acts transitively on the sets $X^k_c:=\{e\in X^k : \delta(e)=c\}$ of monochromatic cells in every dimension $k$ for every $c\in \delta(X^k)$.
\end{definition}

In this paper we are interested in the cases $n=1$ and 
$n=2$.  Namely, we use completely homogeneous $(2,m)$-colourings  in dimension 1 to construct completely homogeneous $(3,3,m)$-colourings in dimension 2, for every  $m\geq 3$.

\begin{theorem}\label{T - main homogeneous colourings} 
Let $m\geq 3$. Suppose we are given:

\begin{enumerate}
\item  three real numbers $\alpha,\beta,\gamma> 0$ such that $\alpha+\beta+\gamma=\pi$, and three metric graphs $L_1$, $L_2$, et $L_3$ of order $m$  and respective edge length $\alpha$, $\beta$, and $\gamma$, each having girth at least  $2\pi$;
\item a completely homogeneous $(2,m)$-colouring on these graphs, with a common set of colours on the edges, and such that the set of colours on vertices on  $L_k$ is  $\{1,2,3\}\setminus \{k\}$.
\end{enumerate}  
 Then there exists a unique $(3,3,m)$-coloured complex $X$, with vertex colour set $\{1,2,3\}$, for which $L_k$ is coloured isomorphic to the link of every vertex of $X$ of colour $k$. Furthermore, this complex is completely homogeneous. 
\end{theorem}

\begin{proof}

We let $G_k$ denote the group of coloured automorphisms of the graph $L_k$. By our assumption on $L_k$, the action of $G_k$ on $L_k$ admits two orbits of vertices and $n$ orbits of edges. 

We establish some claims first.

\begin{claim}
$L_k$ is bipartite.  
\end{claim}

\begin{proof}
This follows since the extremities of an edge contain both vertex colours. 
\end{proof}

\begin{claim} $G_k$ acts freely on the sets of vertices. 
\end{claim}

\begin{proof}
The stabilizer of a vertex is trivial since the edges incident to a vertex have pairwise distinct colours.   
\end{proof}

\begin{claim} \label{C - claim 3}
$G_k$ acts transitively on the set of edges of a single colour. 
\end{claim}

\begin{proof}
This follows since $L_k$ is bipartite and $G_k$ acts transitively on the set of vertices of a single colour.  
\end{proof}

\begin{claim} Every element stabilizing an edge of $G$ is an involution. 
\end{claim}

\begin{proof}
If $\theta$ stabilizes an edge, then $\theta^2$ fixes the end points of this edge, and therefore $\theta^2=\Id$ by the previous lemma. 
\end{proof}

Suppose we are given three completely homogeneous graphs as in (2) and let us construct $X$.

The first step is the construct a ball $B_1$ of simplicial radius 1 with links isomorphic to $L_1$ as a graph. We assign to the faces of $B_1$ the colour of the edges in $L_1$ they correspond to. We give to the center of $B_1$ colour 1, and to every vertex $y$ incident to $x$ the colour of the vertex in $L_1$ associated with the edge $[x,y]$ in $B_1$. We obtain in this way a completely coloured complex $B_1$, by associating to every edge in $B_1$ the unique pair in $\{\{1,2\},\{2,3\},\{1,3\}\}$ corresponding to the colour of its extremities.

We proceed by induction. Suppose a completely coloured simplicial ball $B_n$ is constructed in such a way that it satisfies the assumptions of the theorem in its interior vertices. We write $S_n$ for its boundary and call  $\tilde B_n$ the coloured complex obtained by adding  $q$ coloured triangles above every edge in $S_n$. 

\begin{lemma}
The link of a vertex $y$ of $S_n$ in $\tilde B_n$ is a metric tree $T_y$ which is completely coloured of diameter $\leq 5\pi/3$.
\end{lemma}

\begin{proof}  Since $B_n$ is a simplicial ball, the intersection of $T_y$ with $B_n$ has diameter  $\leq \pi$ (the diameter exactly $\pi$ at the vertices in $\tilde B_{n-1}$ and $2\pi/3$ elsewhere). Adding the faces of $\tilde B_n$, the diameter is now at most $5\pi/3$. Finally, if $y$ has colour $k$, the colour of a vertex in $T_y$ is the unique $k'$ such that $\{k,k'\}$ is the colour of the corresponding edge in $\tilde B_{n}$.  
\end{proof}

\begin{lemma}
Every completely coloured metric tree (with edges of length 1) of diameter $\leq 5\pi/3$ having the same  vertex and edge colours as $L_k$  embeds in  $L_k$ in a colour preserving way  (for every $k=1,2,3$).
\end{lemma}

\begin{proof}
Since the girth of $L_k$ is at least $2\pi$, it follows by assumption that the balls of diameter $5\pi/3$ are coloured trees, with $m$ edges adjacent to every vertex. This shows every completely coloured tree (with the same sets of colours) embeds.
\end{proof}

Let $B_{n+1}$ denote the simplicial ball obtained from $\tilde B_n$ by fixing a completion $\p_y\colon T_y\to L_k$ for every vertex $y$ of $S_n$ of colour $k$, adding the coloured cells corresponding to $L_k\setminus \p_y(T_y)$.  

\begin{claim}
The direct limit $X$ of the $B_n$ is a $(3,3,m)$-coloured complex with vertex colour set $\{1,2,3\}$, for which $L_k$ is coloured isomorphic to the link of every vertex of $X$ of colour $k$.
\end{claim}

\begin{proof}
This is clear since the requirements are local and satisfied by $B_n$ viewed as a subcomplex of $B_{n+1}$.
\end{proof}

Let us now turn the uniqueness of $X$ up to isomorphism. 

\begin{lemma} Suppose $\p_1$ and $\p_2$ are two completely coloured isometric embeddings of a metric completely coloured tree of diameter  at most  $5\pi/3$ (with edges of length $\pi/3$ in $L_k$. Then $\p_2\circ\p_1^{-1}$ extends in a unique way to a completely coloured isometric isomorphism of  $L_k$ (for every $k=1,2,3$).
\end{lemma}

\begin{proof} It follows by  Claim \ref{C - claim 3} that the balls of radius exactly $5\pi/3$ in $L_k$ are pairwise coloured isometrically isomorphic. Since the codomains of $\p
_1$ and $\p_2$ are coloured isometrically isomorphic and included in such balls, the map
 $\p_2\circ\p_1^{-1}$ extends in a unique way.
\end{proof}

In order to show that two complexes $X$ et $X'$ as described in the theorem are coloured isometrically isomorphic, one then proceeds by induction, by showing that the balls $B_n(x)$ et $B_n(x')$ of radius $n$, where $x\in X$ and $x'\in X'$ are two vertices of the same colour, are uniquely coloured isometrically isomorphic. This results from the previous lemma.  
\end{proof}

Finally, uniqueness implies that $X$ is completely homogeneous. More precisely, for any pair of vertices $(x,y)$, there exists a unique colour preserving isometric isomorphism of $X$ taking $x$ to $y$.

\section{Main theorem with signs}\label{S- signs}
 
 The third step in the construction of $G$ and $X$ is to establish a version of Theorem \ref{T - main} ``with signs''.   The signs  (defined below) appear the vertices in a fundamental domain and are necessary for the group $G$ and the complex $X$ to be well-defined in general. 
  There are three choices for the signs for a total of eight possibles constructions $(\e_1,\e_2,\e_3)$, where $\e_i \in \{\pm\}$, for every fixed choice of Sidon sequences and  bijections $\sigma_i$. 

 In the formulation with signs, the puzzle problem is slightly more elaborate, because it must account for the fact that the rings must have markings on their vertices in addition to having  markings on edges as in the situation described in Theorem \ref{T - main}.

\begin{theorem}\label{T - main with signs}
For every  integer $n\geq 2$, every triple of increasing sequences
\[
0\leq a_0^j < a_1^j<\ldots <a_n^j, \ \ j=1,2,3,
\]
every triple of integers $r_j$, $j=1,2,3$,
such that for every $j\in \{1,2,3\}$, the sequence $a_0^j < a_1^j<\ldots <a_n^j$ is a Sidon sequence modulo $r_j$, and every family of bijections
\[
\sigma_j\colon \{a_0^j,\ldots, a_n^j\} \to \{0,\ldots , n\}, \ \ j=1,2,3
\] 
and
\[
\tau_j\colon \{0,1\} \to \{1,2,3\}\setminus \{j\}, \ \ j=1,2,3
\] 
there exists a countable group $G$, acting properly discontinuous on a CAT(0) complex $X$ of dimension 2 with compact quotient,  a $G$-invariant labelling $X^2\to \{0,\ldots,n\}$ of the faces of $X$, and a $G$-invariant labelling $X^1\to \{1,2,3\}$ of the edges of $X$, such that the following holds: if a flat plane embeds in $X$, then it is the solution, with respect to the induced labelling, of a triangle ring puzzle problem with  rings of the form: 
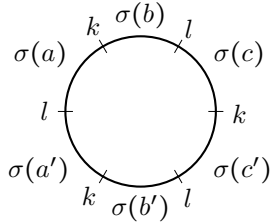
\begin{figure}[H]
 \begin{tikzpicture}
    \tikzset{narrow/.style={inner sep=3pt}}
    \draw [thick] (0,0) circle[radius=1cm];
    \foreach \ang/\lab [evaluate=\ang as \anc using \ang+180]
        in {30/$\sigma(c)$, 90/$\sigma(b)$, 150/$\sigma(a)$, -150/$\sigma(a')$, -90/$\sigma(b')$,-30/$\sigma(c')$}
        \draw [ultra thin] (\ang+30:.9)--(\ang+30:1.1) node{} (\ang:1) node[narrow, anchor=\anc]{\lab};
        \foreach \ang/\lab [evaluate=\ang as \anc using \ang+180]
        in {30/{$l$}, 90/{$k$}, 150/{$l$}, -150/{$k$}, -90/{$l$},-30/{$k$}}
        \draw [ultra thin, anchor=\anc] (\ang+35:1.1) node{\lab};
    \end{tikzpicture}
\end{figure}
\noindent such that $\sigma=\sigma_j$, $l=\tau_j(0)$, $k=\tau_j(1)$, $j\in \{1,2,3\}$,  for every triples $(a,b,c)$ and $(a',b',c')$ of elements of $\{a_0^j,\ldots, a_n^j\}$ with equal alternating sum
\[
a-b+c=a'-b'+c' \mod r_j
\]
such that $a\neq b\neq c$, $a'\neq b'\neq c'$ and $(a,b,c)\neq (a',b',c')$.
\end{theorem}

\begin{proof}
We apply Theorem \ref{T - main homogeneous colourings} with $\alpha=\beta=\gamma=\pi/3$. Consider the three graphs $S_{r_j}(a_0^j, a_1^j,\ldots ,a_n^j)$ for $j=1,2,3$. It is easy to verify that the colouring $\delta_j$ defined by 
\begin{enumerate}
\item $\delta_j(v)=\tau_j(v \mod 2)$ for every vertex $v$;
\item $\delta_j(e)=\sigma_j(a)$ for every edge $e$  labelled by $a\in \{a_0^j, a_1^j\ldots ,a_n^j\}$.
\end{enumerate}
is completely homogeneous. Since the sequences $a_0^j, a_1^j,\ldots ,a_n^j$ are Sidon, the graphs $S_{r_j}$ have girth at least $2\pi$. This shows Theorem \ref{T - main homogeneous colourings} applies. We let $X$ be the corresponding $(3,3,m)$-coloured complex and $G$ the group of completely coloured automorphisms of $X$. Using the maps $\sigma_j$ we obtain  a $G$-invariant labelling $X^2\to \{0,\ldots,n\}$ of the faces of $X$, and using the maps $\tau_j$ we obtain a $G$-invariant labelling $X^1\to \{1,2,3\}$ of the edges of $X$. It remains to prove the assertion on the structure of the space of flats. 

We first show that every flat which embeds in $X$ is the solution of a ring puzzle problem. 
A ring of $X$ is a cycle of length $2\pi$ included in a vertex link of $X$. Every ring is endowed with the induced vertex and edge labeling from $X$. Since $S_{r_j}(a_0^j, a_1^j,\ldots ,a_n^j)$ is bipartite, the vertex labels of a ring at a vertex of type $j$ alternate $k$ and $l$ where $l=\tau_j(0)$ and $k=\tau_j(1)$.  

We must show that the edge labelling has the given form, for some triples $(a,b,c)$ and $(a',b',c')$ of elements of $\{a_0^j,\ldots, a_n^j\}$ with equal alternating sum
\[
a-b+c=a'-b'+c' \mod r_j
\]
such that $a\neq b\neq c$, $a'\neq b'\neq c'$ and $(a,b,c)\neq (a',b',c')$. Let $R$ be a ring in $S_{r_j}(a_0^j, a_1^j,\ldots ,a_n^j)$. We fix an arbitrary base vertex $m\in \IN$ of $R$ with label $l$ and consider the two  triples $(a,b,c)$ and $(a',b',c')$ of elements of $\{a_0^j,\ldots, a_n^j\}$ describing the edge labelling of $R$. Then by the definition of $S_{r_j}(a_0^j, a_1^j,\ldots ,a_n^j)$ the following integers:
\[
m+(2a-1)-(2b-1)+(2c-1) \text{ and } m+(2a'-1)-(2b'-1)+(2c'-1)
\] 
must coincide modulo $2r_j$. This shows that
\[
2a -2b +2c =  2a' -2b' + 2c' \mod 2r_j
\]
which implies the desired result. It is clear conversely that if this condition is satisfied, then 
 the two integers:
\[
m+(2a-1)-(2b-1)+(2c-1) \text{ and } m+(2a'-1)-(2b'-1)+(2c'-1)
\] 
 coincide modulo $2r_j$ for every $m\in \IN$ with label $l$. The corresponding paths in $S_{r_j}(a_0^j, a_1^j,\ldots ,a_n^j)$ are distinct assuming that $(a,b,c)\neq (a',b',c')$, and therefore they define a cycle, which is a ring with the correct labelling. 
 \end{proof}

As mentioned above, we may encode the bijections $\tau_j$ by a sign $\e_j\in \{\pm\}$ at every vertex, where by convention $\e_j=+$ if and only if the bijection $\tau_j$ is increasing. Given the Sidon sequences and the bijections, the above theorem construct eight complexes $X_{(\e_1,\e_2,\e_3)}$. In the next section we prove these complexes are pairwise isomorphic.  

\section{Existence of polarities}\label{S - Polarity}

The fourth step concerns the automorphism group of $X:=X_{(\e_1,\e_2,\e_3)}$. We prove the existence of sufficiently many ``polarities''. This implies that
\[
X_{(\e_1,\e_2,\e_3)}\simeq X_{(\e_1',\e_2',\e_3')}
\]
for every $\e_1,\e_2,\e_3, \e_1',\e_2',\e_3'\in \{\pm\}$.

\begin{definition} Let $S$ be a bipartite graph. 
A \emph{polarity} of $S$ is an automorphism of $S$ of order 2 which permutes the vertex types of $S$.    
\end{definition}

If $e$ is an edge of $S$, we call \emph{polarity at $e$} a polarity which permutes the extremities of $e$.

In the next two lemmas we let $a_0<\ldots<a_n$ be a Sidon sequence modulo $N$ and consider the quotient graph $S_N$ of $S=S(a_0,\ldots,a_n)$ by the action of $2N\ZI$ by translations. 

We assume $S_N$ is endowed with a fixed edge labelling
\[
\sigma\colon \{a_0,\ldots, a_n\} \to \{0,\ldots , n\}
\] 
(as in Theorem \ref{T - main with signs}).

\begin{lemma}\label{L - polarity construction}
$S_N$ admits an edge label preserving polarity at every edge. 
\end{lemma}

\begin{proof} For every $0\leq p\leq n$ consider the map $\p_p\colon S\to S$ taking $k$ to $-k+2a_p-1$. For every $0\leq q\leq n$ and $l$ even we have
\[
\p_p([l,l+2a_q-1])=[-l+2a_p-1,-l+2a_p-1-(2a_q-1)].
\]
This shows $\p_p$ takes edges to edges and defines an automorphism of $S$. This automorphism factorizes to an automorphism of $S_N$, which is a label preserving polarity that exchanges the end vertices of edge $[0,2a_p-1]$. By transitivity of the group of edge label preserving automorphisms, such a polarity exists at every edge of $S_N$.
\end{proof}

\begin{lemma}\label{L - tree isom link}
Let $T,T'$ be isomorphic subtrees of $S_N$ containing at least a tripod, and let  $\p_0\colon T\to T'$ be a isomorphism preserving the edge labels. Then $\p_0$ extends in a unique way to an automorphism of $S_N$ preserving the edge labels.   
\end{lemma}

\begin{proof}
If $T$ and $T'$ are adjacent tripod it suffices to choose the polarity at the common edge. If $T$ and $T'$ are arbitrary tripods, one can use transitivity of the group of automorphisms on vertices of the same type. If $T$ and $T'$ are arbitrary trees, one may fix a tripod $T_0\subset T$ can consider the unique edge label preserving automorphism of $S_N$ whose restriction to $T_0$ is $\p_0$. It takes $T$ to $T'$. Uniqueness is clear since we have required the automorphisms to preserve be the edge labelling.     
\end{proof}

\begin{theorem}
There exists an isomorphism $X_{(\e_1,\e_2,\e_3)}\simeq X_{(\e_1',\e_2',\e_3')}$ which preserve the labels on faces, for every  $\e_1,\e_2,\e_3, \e_1',\e_2',\e_3'\in \{\pm\}$.
\end{theorem}

\begin{proof}
We write $X=X_{(\e_1,\e_2,\e_3)}$ and $X'= X_{(\e_1',\e_2',\e_3')}$. By symmetry, it is enough to consider the case $\e_1\neq \e_1'$, $\e_2=\e_2'$,  $\e_3=\e_3'$. We let $x$ and $x'$ respectively denote a vertex of type 1 in $X$ and $X'$. By Lemma \ref{L - polarity construction}, there exists an edge label perserving isomorphism of the corresponding links. This isomorphism induces an isomorphism $\p_1\colon B_1(x)\to B_1(x')$ between the ball of radius 1, which preserves the labels on the faces. Then, it follows by Lemma \ref{L - tree isom link} that $\p_1$ extends in a unique way to label preserving isomorphisms $\p_n$ between successive balls $B_n(x)$ and $B_n(x')$. The direct limit of these maps is an isomorphism which preserve the labels of the faces.  
\end{proof}
\section{Every ring puzzle embeds in $X$}\label{S - puzzles embed}

The fifth and last step is to prove that every solution to the given ring puzzle problem defines a flat in $X$ (showing (b) $\impl$ (a) in Th.\ \ref{T - main}, which is the last  assertion that remains to be established).

\begin{theorem}
Suppose $\Pi$ is a labeled flat obtained as a solution of the puzzle problem described in Th.\ \ref{T - main}. Then $\Pi$ embeds in $X$ in a label preserving way.
\end{theorem}

 In fact, every  labeled preserving embedding of a 1-ball of $\Pi$ into $X$ extends uniquely to an embedding of $\Pi$ into $X$.

\begin{proof}
Let $\Pi$ be a coloured  flat obtained as a solution of the puzzle problem described in Th.\ \ref{T - main}, and let  $B_1$ be a 1-ball of $\Pi$. Consider a  labeled preserving embedding  $\p_1$  of $B_1$  into $X$. We show that $\p_1$ extends uniquely to an embedding of $\Pi$ into $X$.

The proof is by induction using the following lemma.

\begin{lemma}\label{L - length 4 paths embed}
Let $R$ be a ring of type $j=1,2,3$ (as described in Theorem \ref{T - main}) and let $P\subset R$ be a segment in $R$. Then every label preserving embedding $\psi \colon P\to L_j$ extends to a labeled preserving embedding $\psi'\colon R\to L_j$.   
\end{lemma}

\begin{proof}[Proof of Lemma \ref{L - length 4 paths embed}]
By assumption there exists an edge label preserving embedding $\psi_0\colon R\to L_j$. We let $e$ denote the initial edge in $P$. Since $L_j$ is completely transitive, there exists an edge label preserving automorphism $\p \colon L_j\to L_j$ taking $\psi_0(e)$ to $\psi(e)$. Up to composing with a polarity of  $L_j$ fixing $\psi(e)$, we may assume that $\theta$ takes $\psi_0(f)$ to $\psi(f)$, where $f$ is the edge adjacent to $e$ in $P$ (assuming that $P$ has length at least 2). It follows by the fact that $\theta$ is labeled preserving that $\theta$ takes $\psi_0(P)$ to $\psi(P)$. Then $\psi':=\theta\circ \psi_0$ is a labeled preserving embedding extending $\psi$.  
\end{proof}

Suppose $\p_n$ is an embedding the $n$-ball of $\Pi$ concentric to $B_1$ into $X$ and let us show $\p_n$ can be extended to an embedding  of the $(n+1)$-ball $B_{n+1}$. For every vertex $x$ of $S_n=\del B_n$, we let $R_x$ denote the ring at $x$ in $\Pi$ and $P_x\subset R_x$ denote the path in $R_x$ associated to $B_n$. Then $\p_n$ induces a labeled preserving embedding $\psi_x\colon P_x\to L_{\p_n(x)}$. 
  By the previous lemma, $\psi_x$ extends uniquely to a label preserving embedding $\psi_x'\colon R_x\to L_{\p_n(x)}$. We let $B_{n+1}$ be the union of 
  $B _n$ and $\bigcup_{x\in S_n}  [\psi_x'(R_x)]$, where $[\psi_x'(R_x)]$ is the 1-disk in $X$ corresponding to $\psi_x'(R_x)$. Since the maps $\psi_x'$ are colour preserving, this is a disk of radius $n+1$ in $X$ which extends $B_n$. Thus, $\p_n$ extends in a unique way.
\end{proof}

\section{Examples and applications}\label{S - examples} 

\subsection{Mian--Chowla complexes}

Let $(a_n)_{n\geq 0}$ denote the Mian--Chowla sequence (obtained from $a_0=0$ by the greedy algorithm).

Theorem \ref{T - main} associates groups and complexes to the truncated Mian--Chowla sequence $a_0,\ldots, a_n$ and integer $N$ such that  $a_0,\ldots, a_n$ is a Sidon sequence modulo $N$. In general, this holds for every $N$ large enough:

\begin{proposition}
Let $a_0,\ldots, a_n$ be a Sidon sequence. Then $N_0=2a_n+1$ is the smallest integer  with the property that $a_0,\ldots, a_n$ Sidon sequence modulo $N$ for every $N\geq N_0$.
\end{proposition}

\begin{proof}
It is clear that $N_0\geq 2a_n+1$, since $a_n+a_n=a_0+a_0\mod 2a_n$.
Conversely, if $a_1,\ldots, a_n$ is not Sidon modulo $N$, then there exist $a_i,a_j,a_p,a_q\in \{a_1,\ldots, a_n\}$ such that 
\[
a_i+a_j\neq a_p+a_q
\]
and
\[
N \text{ divides } a_i+a_j-a_p-a_q.
\]
Since $0\leq a_i+a_j, a_p+a_q\leq 2a_n$, this fails if $N\geq 2a_n+1$.
\end{proof}

While the proposition shows the general bound $2a_n+1$ is sharp, we note that there exist Sidon sequences which are Sidon sequence modulo $N$ for some values of $N$ which are smaller than $N_0$. Consider for example the Sidon sequence $0,2,7$. Then $N_0=15$, and it is easily verified that $0,2,7$ is a Sidon sequence modulo $N$, for several values of $N<N_0$ (e.g., $N=8$).

\begin{question}
Given a Sidon sequence $a_0,\ldots a_n$, what is the value of 
\[
N_{00}(a_0,\ldots a_n):=\min\{N : a_0,\ldots a_n\text{ is a Sidon sequence modulo } N\}?
\]
\end{question}
In some cases $N_{00}=N_0$. Consider for example the Sidon sequence $0,1,3$. Then $N_0=7$, and it is easily verified that the Mian--Chowla sequence $0,1,3$ fails to be a Sidon sequence modulo $N$, for every $N\leq 6$. On the other hand, it is not difficult to check, for example, that the Mian--Chowla sequence $0,1,3,7,20$ is a Sidon sequence modulo 35.

For every $n\geq 2$ we call Mian--Chowla complex the CAT(0) 2-complex $X_n$ associated by Theorem \ref{T - main}with the following data:
\begin{enumerate}
\item $a_0^j,\ldots, a_n^j$ is the truncated Mian--Chowla sequence;
\item $\sigma_j\colon \{a_0^j,\ldots, a_n^j\}\to\{0,\ldots, n\}$ is the increasing bijection; 
\item $N_j= N_{00}(a_0^j,\ldots, a_n^j)$;
\end{enumerate}
for every $j=1,2,3$.

Due to the symmetry in these data, we have:

\begin{proposition}
The automorphism groups of the Mian--Chowla complexes are vertex transitive. 
\end{proposition}

It would be interesting to study the geometric structure of the Mian--Chowla complexes $X_n$, and solve the associated ring puzzle problems. We observe that the Mian--Chowla complex $X_2$ is in fact a Bruhat-Tits building.

\subsection{Moebius--Kantor complexes}\label{S - MK odd} We say that a CAT(0) 2-complex $X$ is a \emph{Moebius--Kantor complex} if its faces are equilateral triangles and its vertex links are isomorphic to the Moebius--Kantor graph (namely, the unique bipartite cubic symmetric graph on 16 vertices).

Consider the Sidon sequence $a_0=0$, $a_1=1$, $a_2=3$ (of length 3) modulo $N=8$, and the  bijections 
\[
\sigma_i\colon \{0,1,3\}\to \{0,1,2\}, \ \ i=1,2,3
\]
given by $\sigma_1(0):=0$, $\sigma_1(1):=1$, $\sigma_1(3):=2$, and $\sigma_i:=(0\ 1\ 2)^i\sigma_1$ for $i=2,3$. Applying Theorem \ref{T - main}, we find a group $G$ and a CAT(0) complex $X$ with the indicated properties.

We claim:

\begin{proposition}\label{P - X is MK}
	$X$ is a Moebius--Kantor complex.
\end{proposition}

\begin{proof} By Theorem \ref{T - main}, the space $X$ is a CAT(0) space with equilateral triangle faces. 
The links in $X$ are determined by the Sidon sequence. By definition, they are isomorphic to the graph with vertex set $\ZI/16\ZI$ and edge set $[n,n+1]$ for every $n$ and $[n,n+5]$ for every $n$ even. It is easy to verify that this graph is the Moebius--Kantor graph.	 
\end{proof}

Thus, Th.\ \ref{T - main} provides a new construction method for Moebius--Kantor complexes. Here we shall describe here some properties of $X$, and in particular motivate our choice of  bijections.

We call \emph{root} of $X$ an isometric embedding $\alpha$ of a path of length $\pi$ in a link of $X$, such that $\alpha(0)$ is a vertex. Thus, the image of every root consists of three edges, which we shall endow with the induced labeling in $\{0,1,2\}$. We say that $\alpha$ is a root of rank 2 if there exist precisely two roots distinct from $\alpha$ with the same end points. This definition is a particular case of the notion of rank of a root in a CAT(0) 2-complex---and we shall refer the interested reader to \cite[\ts 4]{chambers} for this generalization.

We write $S_i$ for the link in $X$ at a vertex of type $i$. For the complex $X$, the rank of a root in $S_i$ is a function of its labels; the following statement is the most relevant for our purpose.

\begin{lemma}
	Let $\alpha$ be a root in $S_i$,  $i=1,2,3$, with consecutive labels $b$, $a$, and $b$ where $a\neq b$. Then $\alpha$ is a root of rank 2 if and only if $a=\sigma_i(0)$.
\end{lemma}

\begin{proof}
Suppose an edge $e=[r,r+1]$ has label $a=\sigma_i(0)$ and let $b\neq a$. Then, by our definition of $S_i$, $r$ is odd. Furthermore, since they have the same label $b$, then the two edges adjacent to $e$ in $\alpha$ have the same increment, which is  either 1 or 5. It is easily seen that $\alpha$ is of rank 2 in both cases. 

Suppose now an edge $e=[r,r+1]$ has label $a=\sigma_i(1)$ or $a=\sigma(3)$ and let $b\neq a$. By the same argument, the two edges adjacent to $e$ in $\alpha$ have the same increment. It is easily seen that $\alpha$ is not of rank 2 in both cases. 
\end{proof}

Let $t$ be a triangle in $X$. For every side of $t$, choose a  triangle in $X$ adjacent to $t$. This defines three roots in the links of the vertices of $t$. We say that  $t$ is \emph{odd}  if the number of such roots of rank 2 is odd (this is well-defined by \cite[\ts 2]{parity}). We say that $X$ is \emph{odd} if every triangle is odd.

\begin{proposition}
	$X$ is odd.
\end{proposition}

\begin{proof}
	Consider a triangle $t$ in $X$ with label $a\in \{0,1,2\}$. Let $b\in \{0,1,2\}$, $b\neq a$. Adjacent to $t$ are three triangles with label $b$. These triangles form, together with $t$, a larger triangle which we call $T$. Suppose $a=\sigma_1(k)$ for some $k\in\{0,1,3\}$. Then $\sigma_2(k)=(0\ 1\ 2)a$, $\sigma_3(k)=(0\ 1\ 2)^2a$; therefore $\sigma_2((0\ 1\ 3)k)=a$ and $\sigma_3((0\ 1\ 3)^2k)=a$. This implies that $a=\sigma_i(0)$ for a unique $i=1,2,3$ which in turns shows $T$ contains a single roots of rank 2. This proves that $t$ is odd.   
\end{proof}

\subsection{Modular complexes} Let $a_0=0<a_1<\ldots <a_n$ be a sequence and  $N_1$, $N_2$, and $N_3$ be integers such that $a_0=0<a_1<\ldots <a_n$ satisfies the condition of Sidon modulo $N_i$ for $i=1,2,3$. Let $G(a_0,a_1,\ldots ,a_n :N_1,N_2,N_3)$ and, respectively, $X(a_0,a_1,\ldots ,a_n: N_1,N_2,N_3)$ denote be the group and complex obtained by applying Th.\ \ref{T - main} with respect to these data and the increasing bijection $\sigma\colon  \{a_0,\ldots, a_n\} \to \{0,\ldots , n\}$.

To illustrate, we explain how Th.\ \ref{T - main} provides an alternative approach to \cite[\ts 13]{surgery} for a construction ``mixing'' the Moebius--Kantor local geometry to that of $\tilde A_2$ buildings in a same complex. Such a complex was said to be ``of strict type $A_\mathrm{MK}+\tilde A_2$''. It was obtained in \cite{surgery} by a surgery construction, relying on the classification of collars between two ``partial complexes'' of both types. Here we can obtain similar results in as a direct consequence of Theorem \ref{T - main}. For instance, using the terminology of \cite{surgery}, we claim:

\begin{proposition}[Compare {\cite[Prop.\ 13.1]{surgery}}] The modular complex $X(0,1,3: 7,7,8)$ is of strict type $A_\mathrm{MK}+\tilde A_2$.
\end{proposition}

The proof follows as in Prop.\ \ref{P - X is MK} above. Similarly, the complex $X(0,1,3: 7,8,8)$ is a complex of strict type $A_\mathrm{MK}+\tilde A_2$, and it is not isomorphic to $X(0,1,3: 7,7,8)$. These ``modular complexes'' seem particularly interesting  when $N_i$ is small relative to $a_n$. Furthermore,  one can use bijections other than $\sigma$ to twist the construction of modular complexes (for example, the complex $X$ described in \ts\ref{S - MK odd} can be viewed as a ``twisted modular complex''). We define ``the'' modular complex to be as untwisted as possible:

\begin{definition}\label{D- modular complex}
Let $a_0=0<a_1<\ldots <a_n$ be a Sidon sequence. We let 
\[
X(a_0,\ldots, a_n):=X(a_0,a_1,\ldots ,a_n: N_{00},N_{00},N_{00}),
\] 
where $N_{00}:=N_{00}(a_0,\ldots, a_n)$, and call this complex \emph{the modular complex} associated with $a_0=0<a_1<\ldots <a_n$.
\end{definition}

By definition, the Mian--Chowla complexes are modular complexes in this sense. Due to  symmetry in the data, follows by Theorem \ref{T - main} that the automorphism group of the modular complex $X(a_0,\ldots, a_n)$ is transitive on the vertex set. (This would not be true of  generalized modular complex, for example, $X(0,1,3: 7,7,8)$ is not vertex transitive.)

\begin{remark} These examples can be further generalized.  If $A$ is a finite abelian group, one says a set $\{a_0,a_1,\ldots, a_n\}$ in $A$ is Sidon if the number of pairs of elements in $\{a_0,a_1,\ldots, a_n\}$ with a given sum is at most two. It is not difficult to extend our results to such Sidon sets; this gives additional, natural  generalizations of our Sidon complexes.
\end{remark}

It would be interesting to study the asymptotic properties of the $X(a_0,\ldots, a_n)$ complexes, for a fixed infinite Sidon sequence $a_0,a_1,a_2,\ldots$, including for example, the Mian--Chowla sequence or the Rusza sequence.  We shall not pursue this direction on the present occasion, and conclude this paper with an application to the study of Moebius--Kantor complexes.

\subsection{Uniqueness of the odd Moebius--Kantor complex}\label{S - applications} 
In this section we prove that the twisted modular complex constructed in \ts\ref{S - MK odd}  is the unique odd Moebius--Kantor complex up to isomorphism; furthermore, we establish an unique extension theorem for automorphisms. This is done by ``mapping'' the data associated with the Sidon sequence to an arbitrary odd complex, and may compared to \cite{uniqueeven}, in which a similar result is proved in the even case: the even Moebius--Kantor complex is unique up to isomorphism. This was established in \cite{uniqueeven} by ``mapping''   Pauli matrices from the (even) Pauli complex to an arbitrary even complex.

\begin{theorem} Let $X$ and $X'$ be odd Moebius--Kantor complexes and let $x\in X$ and $x'\in X'$ be vertices. Let $B_1(x)$ and $B_1(x')$ denote the ball of radius 1 with center $x$ and $x'$, respectively, and let  $\p_1\colon B_1(x)\to B_1(x')$ be an isomorphism. Then there exists a unique isomorphism $\p\colon X\to X'$ which coincides with $\p_1$ on $B_1(x)$. 
\end{theorem}

\begin{proof}
We may assume that $X$ is the complex constructed in \ts\ref{S - MK odd}; furthermore, by symmetry, we may assume that $x$ is a vertex of type 1 in this complex (associated with the map $\sigma_1$). Let us first extend $\p_1$ to the ball $B_2(x)$ of radius 2. We begin with the following lemma.

\begin{lemma}\label{L - MK tripod} 
Suppose $S$ and $S'$ are Moebius--Kantor graphs, $T\subset S$ and $T'\subset S'$ are tripods, and $\psi_0\colon T\to T'$ is an isomorphism. Then there exists a unique isomorphism $\psi\colon S\to S'$ which coincides with $\psi_0$ on $T$.  
\end{lemma}

\begin{proof} 
Existence follows because the Moebius-Kantor graph is 2-arc-transitive. If $e$ and $f$ are consecutive edges in $T$, then there exists a graph isomorphism $S\to S'$ taking respectively $e$ and $f$ to $\psi_0(e)$ and $\psi_0(f)$. This isomorphism is unique since the stabilizer of a tripod is trivial.
\end{proof}

For every vertex $y$ in the sphere $\del B_1(x)$ of radius 1 centred at $x$, we let $\psi_y$ denote the unique isomorphism  between the ball $B_1(y)$ and the ball $B_1(\p_0(y))$ induced by the previous lemma, which extends $\p_0$ on $B_1(x)$.  

\begin{lemma}\label{L - lemma psi consistent}
The maps $\psi_y$ are consistent. 
\end{lemma}

\begin{proof}
We must show that for every edge $[y,z]$ in $\del B_1(x)$, the maps $\psi_y$ and $\psi_z$ coincide on set of triangles adjacent to $[y,z]$. Since $x$ is of type 1, we may assume that $y$ is of type 2 and $z$ of type 3.  Let $[t_1,y,z]$ be a triangle distinct from $[x,y,z]$, and consider the two triangles $[t_2,x,y]$ and $[t_3,x,z]$ whose labels coincide with that of $[t_1,y,z]$. We write $\alpha_x$, $\alpha_y$ and $\alpha_z$ for the three roots, respectively at $x$, $y$ and $z$, associated with this configuration. There are two cases. 

Suppose first that $[x,y,z]$ is labeled by $a=\sigma_1(0)$. Since $a\neq \sigma_2(0)$ and $a\neq \sigma_3(0)$, both roots $\alpha_y$ and $\alpha_z$ fail to be of rank 2. Since $\p_0$, $\psi_y$ and $\psi_z$ are isomorphism,  $\p_0(\alpha_x)$ is of rank 2 in $X'$, while $\psi_y(\alpha_y)$ and $\psi_z(\alpha_z)$ are not. Since the triangle $\p_0([x,y,z])$ is odd, the two triangles $\psi_y([t_1,y,z])$ and $\psi_z([t_1,y,z])$ must coincide.   

Suppose next that $[x,y,z]$ is labeled by $a=\sigma_2(0)$ or $a=\sigma_3(0)$. The two cases are symmetric and we  assume $a=\sigma_2(0)$ to fix the ideas. Then $\alpha_y$ is of rank 2, while $\alpha_x$ and $\alpha_z$ are not, and the same must be true of their images. Again, since the triangle $\p_0([x,y,z])$ is odd, the two triangles $\psi_y([t_1,y,z])$ and $\psi_z([t_1,y,z])$  coincide.     
\end{proof}

By Lemma \ref{L - lemma psi consistent}, the map 
\[
\p_2:=\p_1\vee \bigvee_{y\in \delta B_1(x)} \psi_y
\]
is well defined. It induces an isomorphism between $B_2(x)$ and $B_2(x')$ which extends $\p_1$ by definition. Furthermore, this extension is unique by Lemma \ref{L - MK tripod}.

Let $n\geq 2$. Let $\p_n\colon B_n(x)\to B_n(x')$ is an isomorphism, and fix a vertex  $y$ in the sphere $\del B_n(x)$ of radius $n$ centred at $x$. If there does not exist a triangle $[y,y_1,y_2]$ in $B_n$ such that $[y,y_1,y_2]\cap \del B_n=\{y\}$, we  let $\psi_y$ denote the unique isomorphism  between the ball $B_1(y)$ and the ball $B_1(\p_n(y))$ induced by the Lemma \ref{L - MK tripod}, which extends $\p_n$ on $B_n(x)$.

Suppose that there exists a triangle $[y,y_1,y_2]$ in $B_n$ such that $[y,y_1,y_2]\cap \del B_n=\{y\}$. Lemma \ref{L - MK tripod} provides two maps $\psi_y^1$ and $\psi_y^2$ between the ball $B_1(y)$ and the ball $B_1(\p_n(y))$ induced by the Lemma \ref{L - MK tripod}, which extends the restriction  $\p_n$ to the set of triangles adjacent to $[y,y_1]$ and $[y,y_2]$, respectively.  We show that:

\begin{lemma} 
$\psi_y^1=\psi_y^2$; furthermore, they extend the restriction of $\p_n$ to $B_n\cap B_1(y)$. 
\end{lemma}

\begin{proof}
Let $a$ denote the label of $[y,y_1,y_2]$. Consider triangles $[t_0,y_1,y_2]$, $[t_1,y,y_1]$ and $[t_2,y,y_2]$ with the same label $b\neq a$, and the corresponding roots $\alpha_y$, $\alpha_{y_1}$ and $\alpha_{y_2}$, respectively. We assume that $y$ is of type 1. The other cases are similar by symmetry.

Suppose $a=\sigma_1(0)$. Then $\alpha_y$ is a root of rank 2, while $\alpha_{y_1}$ and $\alpha_{y_2}$ are not, and $\psi_y^1$ takes $[t_2,y,y_2]$ to the unique triangle in $X'$ such that $\psi_y^1(\alpha_y)$ is a root of rank 2 in the link of $\p_n(y)$. Since the triangle $\p_n([y,y_1,y_2])$ is odd in $X'$, this shows that $\p_n$ and $\psi_y^1$ coincide on $B_n\cap B_1(y)$. By symmetry,  $\p_n$ and $\psi_y^2$ coincide on $B_n$. Since $\psi_y^1$ and $\psi_y^2$ coincide on (at least) a tripod, they must coincide everywhere by Lemma \ref{L - MK tripod}. The two other cases, namely, $a=\sigma_2(0)$ and $a=\sigma_3(0)$, are similar.
\end{proof}

In the case that there exists a triangle $[y,y_1,y_2]$ in $B_n$ such that $[y,y_1,y_2]\cap \del B_n=\{y\}$, we let $\psi_y:=\psi_y^1=\psi_y^2$; this now defines $\psi_y$ for all $y\in \delta B_n(x)$. A direct generalization of Lemma \ref{L - lemma psi consistent} to larger balls  show that the maps $\psi_y$ are consistent.

It follows that 
\[
\p_{n+1}:=\p_n\vee \bigvee_{y\in \delta B_n(x)} \psi_y
\]
is well defined, and induces an isomorphism between $B_{n+1}(x)$ and $B_{n+1}(x')$ which extends $\p_n$ by definition. This extension is unique by Lemma \ref{L - MK tripod}. Thus, $\p:=\dirlim \p_n$ is an isomorphism from $X$ to $X'$ which extends $\p_1$ uniquely.
\end{proof}

\end{document}